\documentclass[12pt]{article}
\usepackage{color}
\usepackage{graphicx}
\usepackage{amsmath}
\usepackage{amssymb}
\usepackage{mathrsfs}
\usepackage[numbers,sort&compress]{natbib}
\usepackage{amsthm}
\usepackage{tikz}
\usepackage{float}
\makeatletter

\newtheorem{lemma}{Lemma}[section]

\newtheorem{thm}{Theorem}[section]
\newtheorem{proposition}{Proposition}

\@addtoreset{equation}{section}
\makeatother

\begin{document}
\title{Long-time asymptotic behavior of a  mixed schr\"{o}dinger equation with weighted Sobolev initial data }
\author{Qiaoyuan Cheng, Yiling Yang and   Engui Fan$^1$\thanks{\ Corresponding author and email address: faneg@fudan.edu.cn } }
\footnotetext[1]{ \  School of Mathematical Sciences  and Key Laboratory of Mathematics for Nonlinear Science, Fudan University, Shanghai 200433, P.R. China.}

\date{ }

\maketitle
\begin{abstract}
\baselineskip=19pt

We apply $\bar{\partial}$  steepest descent method
 to obtain   sharp asymptotics  for a  mixed schr\"{o}dinger equation
$$
q_t+iq_{xx}-ia (\vert q \vert^2q)_{x}-2b^2\vert q \vert^2q=0,
$$
$$
q(x,t=0)=q_0(x),
$$
under essentially minimal regularity
assumptions on   initial data   in a   weighted Sobolev space $q_0(x) \in H^{2,2}(\mathbb{R})$.
In the asymptotic expression,    the  leading  order  term  $\mathcal{O}(t^{-1/2})$ comes   from   dispersive  part $q_t+iq_{xx}$   and    the  error order
 $\mathcal{O}(t^{-3/4})$   from a $\overline\partial$ equation.
\end{abstract}
\baselineskip=20pt

\newpage
\section{Introction}
\hspace*{\parindent}

 In this paper, we consider a  mixed nonlinear Schr\"{o}dinger (NLS) equation with  a usual cubic nonlinear term and a
derivative cubic nonlinear term
\begin{align}
& q_t+iq_{xx}-ia (\vert q \vert^2q)_{x}-2b^2\vert q \vert^2q=0, \label{MNLS}\\
& q(x,t=0)=q_0(x)\in H^{2,2}(\mathbb{R}),\label{MNLS2}
\end{align}
where
$$H^{2,2}(\mathbb{R})=\{f\in L^2(\mathbb{R}): x^2f,f''\in L^2(\mathbb{R})\}. $$
The equation (\ref{MNLS})  was presented and further solved by the inverse scattering method \cite{ago1979,imf1980} and can be used to describe Alfv\'{e}n waves propagating along the magnetic field in cold plasmas and the deep-water gravity waves \cite{Mio1,Stiassnie1}. The term $i(|q|^2q)_x$ in the equation (\ref{MNLS}) is called the self-steepening term, which causes an optical pulse to become asymmetric and steepen upward at the trailing edge \cite{Al, Yang}.  The equation (\ref{MNLS}) also describes the short pulses propagate in a long optical fiber characterized by a nonlinear refractive index \cite{Nakatsuka1,Tzoar1}. Brizhik et al showed that the modified NLS equation (\ref{MNLS}), unlike the classical NLS equation, possesses static localized solutions when the effective nonlinearity parameter is larger than a certain critical value \cite{RN13}.

For $a=0$, the equation (\ref{MNLS}) reduces the classical defocusing Schr\"{o}dinger equation
\begin{equation}
iu_t+u_{xx}-2b^2|u|^2u=0,\label{nls}
\end{equation}
which has important applications in a wide variety of fields such as nonlinear optics, deep water waves, plasma physics, etc \cite{RN5,RN6}.

For $b=0$, the equation (\ref{MNLS}) reduces to the Kaup-Newell equation \cite{RN4}
\begin{equation}
iu_t+u_{xx}-ia(|u|^2u)_x=0. \label{kn}
\end{equation}

For $a, b  \not=0$, the equation (\ref{MNLS}) is equivalent to the modified  NLS equation
\begin{equation}
q_t+q_{xx}+ia(|q|^2 q)_x-2|q|^2q=0,  \label{mnls}
\end{equation}	
which was also called the perturbation NLS equation \cite{Maimistov1}.

There are various works on  mixed NLS  equation (\ref{MNLS}) and the modified NLS equation (\ref{mnls}). For example,  Date  analyzed  the periodic  soliton  solutions   \cite{pso1985}. In 1991, Guo   and Tian studied the unique existence and the decay behaviours of the smooth solutions to the initial value problem \cite{oss1991}.
 In 1994, Tian  and Zhang  proved that the Cauchy problem of initial value in Sobolev space has a unique weak solution \cite{ows1994}.  The existence and nonexistence of global solution and blow-up solution  for this equation  were also investigated \cite{bsf2004}.  The bright,  dark envelope solutions and Soliton behavior with N-fold Darboux transformation for  mixed nonlinear Schr\"{o}dinger equation were also discussed \cite{mfd2015,sbf2013,hs2015}.

 The local well-posedness of smooth solutions  in the Sobolev spaces $H^s(\mathbb{R})$, $s>3/2$   was
established by Tsutsumi and Fukuda \cite{osf} and later extended to solutions with low regularity in the Sobolev space $H ^{1/2}(\mathbb{R})$ by Takaoka \cite{ht1999}.
For initial conditions in the energy space $H^{1}(\mathbb{R})$,  Hayashi-Ozawa proved that solutions exist in $H^{1}(\mathbb{R})$\cite{nt1992}.   Colliander et al   extended this result to $u_0\in H^{1/2+\epsilon}(\mathbb{R})$.\cite{jm2002}    More recently,  global  well-posedness of  derivative NLS equation
 with initial conditions $u_0\in H^{1}$ and $H^{1/2}(\mathbb{R})$, respectively  has been discussed by  Guo\cite{zy2017} and Wu \cite{yw2015} .
 According to the existing conclusion, the global well-posedness of DNLS in $H^{2,2}(\mathbb{R})$ has been rigorously proved by using the inverse scattering transformation  \cite{gwp33,tdn}.

In recent years, McLaughlin and Miller developed a $\bar\partial$-steepest decedent method for obtaining  asymptotic  of RH problems  based on  $\bar\partial$-problems \cite{IMRN2006}.  This method  has been  successfully adapted to study the  NLS  equation and derivative NLS equation \cite{lta29,lta31,ANN2018}.
We recently  have applied  this  method to  obtain asymptotic for the  Kundu-Eckhaus equation and modified NLS equation with weighted Sobolev initial data \cite{Ma2019,yf2019}.

Recently,   for  the  mixed defocusing nonlinear Schr\"{o}dinger equation (\ref{MNLS})
with Schwartz initial data $q_{0}(x) \in \mathcal{S}(\mathbb{R})$,  by defining   a     general   analytical   domain and
two reflection  coefficients,    we found   an  unified  long time asymptotic formula via the Deift-Zhou nonlinear steepest descent method \cite{CF}
\begin{align}
&q(x,t)= t^{-1/2}\alpha(z_0)e^{i(4tz_{0}^{2}-\nu(z_0)\log8t)}e^{-\frac{2ia}{\pi}\int^{+\infty}_{\lambda_0}\log(1-|r(\lambda)|^2)(a\lambda-b)\mathrm{d}\lambda}+\mathcal{O}(t^{-1}\log t),\nonumber
\end{align}
which  covers results on classical defocusing Schr\"{o}dinger equation, derivative Schr\"{o}dinger equation and modified Schr\"{o}dinger equation   as special cases.  In this paper,  for      much weaker weighted Sobolev initial data $q_{0}(x) \in H^{2, 2}(\mathbb{R})$,  we   apply  $\bar\partial$   steepest decedent method  to     give   long time asymptotic  for   the
initial value problem  (\ref{MNLS})-(\ref{MNLS2})
\begin{equation}
q(x,t)=t^{-1/2}\alpha(z_0)e^{i(4tz_{0}^{2}-\nu(z_0)\log8t)}e^{-\frac{2ia}{\pi}\int^{+\infty}_{\lambda_0}\log(1-|r(\lambda)|^2)(a\lambda-b)\mathrm{d}\lambda}+O(t^{-3/4}).\nonumber
\end{equation}
The advantages of this method  are that it avoids delicate estimates involving $L^p$ estimates of Cauchy projection operators.  Here we consider  the mixed NLS equation  (\ref{MNLS}) in  a weighted Sobolev space
 $H^{2,2}(\mathbb{R})$  to  global existence  of solutions and provide a long time  asymptotic result in classical sense
 by using inverse scattering theory.

Our  article is arranged as follows.  In section 2 and  section 3,
based on  the Lax pair of  the mixed NLS equation (\ref{MNLS}),   we  set up a RH problem $N(z)$ associated with the initial value problem (\ref{MNLS})-(\ref{MNLS2}),  which will be used
 to analyze   long-time asymptotics  of the mixed NLS equation (\ref{MNLS}) in our paper. In section 4, by defining a transformation $N^{(1)}(z)=N(z)\delta^{-\sigma_3}$,  we get a  RH problem   whose jump matrices admits
 two type of  up and down   triangular decompositions.  In section 5,
 we establish a hybrid $\overline{\partial}$ problem by the transformation $N^{(2)}(z)=N^{(1)}(z)\mathcal{R}^{(2)}$.
 The $N^{(2)}(z)$ is further decomposed into a pure RH problem  for $N^{rhp}(z)$ and a pure $\overline{\partial}$-problem for $E(x,t,z)$,  where the pure RH problem $N^{rhp}(z)$  is a solvable model
 associated with   a  Weber equation;
      the error estimates on the  pure $\overline{\partial}$-problem  will   given  in  In section 6.
      Finally in section 7,  according to a series of transformations made above,
we establish  a relation $N(z)=E(z)N^{rhp}(z)\mathcal{R}^{(2)}(z)^{-1}\delta^{\sigma_3}$ ,  from
which the asymptotic behavior of the mixed NLS equation can be obtained.

\section{Spectral analysis}

\hspace*{\parindent}
The defocusing  mixed NLS equation (\ref{MNLS})  admits  the following   Lax pair \cite{ago1979}
\begin{align}
&\psi_x+i\lambda(a\lambda-2b)\sigma_3\psi=P_1\psi,\label{Lax 1}\\
&\psi_t+2i\lambda^2(a\lambda-2b)^2\sigma_3\psi=P_2\psi, \label{Lax 2}
\end{align}
where
\begin{equation}\label{hfjf}
\begin{aligned}
&P_1=(\lambda a-b)Q, \ \ Q=\begin{pmatrix}
0&q\\
\bar{q}&0
\end{pmatrix},\\
&P_2=-i(a^2\lambda^{2}-2ab\lambda+b^2)Q^{2}\sigma_3+(2a^2\lambda^3-6ab\lambda^2+4b^2\lambda)Q\\
&+i(\lambda a-b)Q_x\sigma_3+(\lambda a^2-ab)Q^3.
\end{aligned}
\end{equation}

\subsection{Asymptotic  }
\hspace*{\parindent}
Due  to $q_0(x)\in H^{2,2}(\mathbb{R})$,  so as $x \rightarrow \pm\infty$,  the  Lax pair (\ref{Lax 1})-(\ref{Lax 2}) becomes
\begin{align}
&\psi_x+i\lambda(a\lambda-2b)\sigma_3\psi\sim 0,\nonumber\\
&\psi_t+2i\lambda^2(a\lambda-2b)^2\sigma_3\psi\sim 0, \nonumber
\end{align}
which implies that the Jost solutions of  the Lax pair (\ref{Lax 1})-(\ref{Lax 2}) admit     asymptotic
\begin{equation}
\psi\sim e^{-it\theta(x,t,\lambda)\sigma_3}, \ \ x\rightarrow \pm \infty,
\end{equation}
where $\theta(x,t,\lambda)= \lambda(a\lambda-2b)x/t +2 \lambda^2(a\lambda-2b)^{2}$.
Therefore,   making transformation
\begin{equation}
\Psi=\psi e^{it\theta(x,t,\lambda)\sigma_3},\label{bianhuan}
\end{equation}
we have
\begin{equation}
\Psi \sim  I, \ x\rightarrow \pm \infty, \label{xdjx}
\end{equation}
and $\Psi$  satisfies a new  Lax pair
\begin{align}
&\Psi_x+i\lambda(a\lambda-2b)[\sigma_3,\Psi]=P_1\Psi, \label{lax3}\\
&\Psi_t+2i\lambda^2(a\lambda-2b)^2[\sigma_3,\Psi]=P_2\Psi. \label{lax4}
\end{align}
We consider asymptotic expansion
\begin{equation}
\Psi=\Psi_0+\frac{\Psi_1}{\lambda}+\frac{\Psi_2}{\lambda^2}+O(\frac{1}{\lambda^{3}}),\qquad \lambda\rightarrow\infty, \label{expan}
\end{equation}
where $ \Psi_0, \Psi_1, \Psi_2 $ are independent of $ \lambda $.
Substituting (\ref{expan}) into (\ref{lax3}) and   comparing  the coefficients of
$\lambda$,  we obtain  that $\Psi_0$ is a diagonal matrix and
\begin{align}
 &\Psi_{1}=\frac{i}{2}Q\Psi_0\sigma_{3},\label{1o} \\
 &\Psi_{0x}+ia[\sigma_3,\Psi_2]=aQ\Psi_{1}+bQ\Psi_0. \label{ree}
\end{align}
In the same way, substituting (\ref{expan}) into (\ref{lax4}) and   comparing  the coefficients of
$\lambda$ in the same order leads to
\begin{align}
&\Psi_{0x}=\frac{i}{2}a|q|^2\sigma_3\Psi_0.\label{jieguo1}\\
&\Psi_{0t}=[\frac{3}{4}ia^2|q|^4+\frac{1}{2}a(\bar q q_{x}-\bar q_{x}q)]\sigma_3\Psi_0.\label{jieguo2}
\end{align}

Noting  that  the mixed NLS equation  (\ref{MNLS}) admits the conservation law
\begin{equation}
\left(\frac{i}{2}a|q|^2 \right)_t=\left[\frac{3}{4}ia^2|q|^4+\frac{1}{2}a(\bar{q}q_x-\bar{q}_xq)\right]_{x},
\end{equation}
so  two equations (\ref{jieguo1}) and (\ref{jieguo2})  are compatible if we define
\begin{equation}
\Psi_0(x,t)=e^{\frac{ia}{2}\int_{-\infty}^x |q(x',t)|^2dx' \sigma_3}.\label{jieguo}
\end{equation}
 We introduce   a new function $\mu=\mu(x,t,\lambda)$ by
\begin{equation}
\Psi(x,t,\lambda)=\Psi_0(x,t)\mu(x,t,\lambda),\label{xx}
\end{equation}
then $\mu$  admits  the asymptotic
\begin{equation}
\mu=I+O(\frac{1}{\lambda}),\quad \lambda\rightarrow\infty,
\end{equation}
and satisfies  the  Lax pair
\begin{subequations}
\begin{align}
&\mu_{x}+i\lambda(a\lambda-2b)[\sigma_{3},\mu]=
e^{-\frac{ia}{2}\int_{-\infty}^x|q|^2(x',t)dx'\widehat{\sigma}_3} H_1 \mu, \label{kk1}\\
&\mu_{t}+2i\lambda^{2}(a\lambda-2b)^{2}[\sigma_{3},\mu]=e^{-\frac{ia}{2}\int_{-\infty}^x|q|^2(x',t)dx'\widehat{\sigma}_3} H_2\mu,\label{kk2}
\end{align}
\end{subequations}
 where
 $$H_1= P_1-\frac{ia}{2}|q|^2(x,t)\sigma_3, \ \   H_2=P_2-\int_{-\infty}^x(q\overline{q}_t+q_t\overline{q})dx'\sigma_3.
 $$

\subsection{Analyticity  and symmetry   }
\hspace*{\parindent}
According to (\ref{1o}) and (\ref{xx}), we can establish the relationship between the solution $q(x,t)$ of the  mixed NLS and $\Psi$ as follows
\begin{equation}
q= 2i e^{\frac{ia}{2}\int_{-\infty}^x |q(x',t)|^2dx'}\lim_{\lambda\rightarrow \infty}(\lambda\Psi)_{12},
\end{equation}
which combines with  (\ref{xx}) gives
\begin{align}
q(x,t)=2ie^{ia\int_{-\infty}^x |q(x',t)|^2dx'}\lim_{\lambda\rightarrow \infty}(\lambda\mu)_{12}.\label{q1}
\end{align}

 The Lax pair  (\ref{kk1})-(\ref{kk2})  can be written into  a  full derivative form
\begin{align}
&d(e^{it\theta(x,t,\lambda)\widehat{\sigma}_3}\mu  )= e^{it\theta(x,t,\lambda)\widehat{\sigma}_3}e^{-\frac{ia}{2}\int_{-\infty}^x|q|^2(x',t)dx'\widehat{\sigma}_3}(H_1 dx +H_2dt)\mu.\label{qwtf}
\end{align}
The two solutions of equation (\ref{qwtf}) are
\begin{align}
&\mu_- =I+\int_{- \infty}^x e^{-i\lambda(a\lambda-2b)(x-y)\widehat{\sigma}_{3}}e^{-\frac{ia}{2}\int_{-\infty}^y|q|^2(x',t)dx'\widehat{\sigma}_3}H_1(y,t,\lambda)\mu_-(y,t)dy,\label{pioi}\\
&\mu_+ =I+\int_{ \infty}^x e^{-i\lambda(a\lambda-2b)(x-y)\widehat{\sigma}_{3}}e^{-\frac{ia}{2}\int_{-\infty}^y|q|^2(x',t)dx'\widehat{\sigma}_3}H_1(y,t,\lambda)\mu_+(y,t)dy.\label{pioi6}
\end{align}
We define two domains by
\begin{align}
D^{+} =\{ \lambda | (a \text{Re}\lambda-b)\text{ Im}\lambda>0\}, \ \  D^{-} =\{ \lambda | (a \text{Re}\lambda-b)\text{ Im}\lambda<0\}, \label{dom}
\end{align}
then   boundary of the  domains   $D^{+}$ and   $D^{-}$ is given by
\begin{equation}
\Sigma= \{ \lambda | (a \text{Re}\lambda-b)\text{ Im}\lambda= 0\}. \label{boun}
\end{equation}
\begin{figure}[H]
\begin{center}
\begin{tikzpicture}
\draw [fill=pink,ultra thick,pink] (1,0) rectangle (3,1.8);
\draw [fill=pink,ultra thick,pink] (1,0) rectangle (-1.3,-1.8 );
\draw [->](-2,0)--(3.5,0);
\draw [->](0,-2)--(0,2);
\draw [ ](1,-2)--(1,2);
 \node at (2, 1 )  {$D^+$};
\node at (4, 0 )  {$\text{Re }\lambda$};
\node at (0, 2.5 )  {$\text{Im }\lambda$};
\node at (1.3, -0.2 )  { $b/a$};
 \node at (-0.6, -1 )  {$D^+$};
\node at (2, -1 )  {$D^-$};
 \node at (-0.6, 1 )  {$D^-$};

\end{tikzpicture}
\end{center}
\caption{ Analytical domains $D^+$, $D^-$ and boundary $\Sigma$ corresponding to the mixed  NLS equation.}
\label{figure1}
 \end{figure}
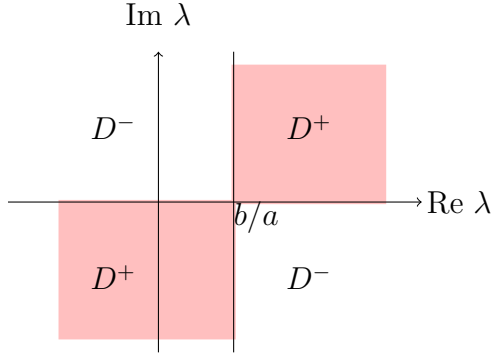

We denote $\mu_\pm=(\mu_{\pm,1}, \mu_{\pm,2})$. Starting from the integrated equation (\ref{pioi})-(\ref{pioi6}), by constructing the  Neumann series, it follows that

\begin{proposition} For  $q(x)  \in L^1(\mathbb{R} )$  and    $t\in \mathbb{R}^+$,  the eigenfunctions defined by (\ref{pioi}) exist and are unique.
Moreover,  $\mu_{-,1}$ and $\mu_{+,2}$ are analytic in $D^{+}$; $\mu_{-,2}$ and $\mu_{+,1}$ are analytic in $D^{-}$.
\end{proposition}

Again by using  transformations (\ref{bianhuan})  and  (\ref{xx}),  we have
\begin{equation}
\mu_-(x,t,\lambda)=\mu_+(x,t,\lambda)e^{-it\theta(x,t,\lambda)\widehat{\sigma}_{3}}S(\lambda).\label{omega}
\end{equation}
where
$$
S(\lambda)=\begin{pmatrix}
s_{11}(\lambda)&s_{12}(\lambda)\\
s_{21}(\lambda)&s_{22}(\lambda)
\end{pmatrix}.
$$
\begin{proposition}  The eigenfunctions  $ \mu_\pm$ and  scattering matrix  $S(\lambda)$   satisfy  the symmetry relations
\begin{align}
&\mu_\pm(x, t, \lambda)=\sigma_1 \overline{\mu_\pm \left(x,t,\overline{\lambda}\right)} \sigma_1, \quad S(\lambda)=\sigma_1 \overline{S(\overline{\lambda})}\sigma_1,\label{dcx1}\\
&\mu_\pm(x,t,\lambda)=-\sigma_* \overline{\mu_\pm\left(x,t,2b/a-\overline{\lambda}\right)} \sigma_*, \quad S(\lambda)=-\sigma_* \overline{S\left(2b/a-\overline{\lambda}\right) }\sigma_*,\label{dcx2}
\end{align}
with
$$\sigma_{*}=\begin{pmatrix}
0 & 1 \\
-1 & 0
\end{pmatrix}, \ \
\sigma_1=\begin{pmatrix}
0 & 1 \\
1 & 0
\end{pmatrix}.
$$
\end{proposition}

\section{The construction of  a  RH problem}
\hspace*{\parindent}
We further use  eigenfunction   $\mu(x,t,z)$ to set up a  RH problem.  Define
\begin{equation}
M(x,t,\lambda)=\begin{cases}
(\frac{\mu_{-,1}}{s_{11}(\lambda)},\mu_{+,2}) ,\quad \lambda\in D^+,\\[6pt]
(\mu_{+,1},\frac{\mu_{-,2}}{s_{22}(\lambda)}),\quad\quad \lambda\in D^-,
\end{cases}
\end{equation}
then equation (\ref{omega})   leads to   the following Riemann-Hilbert problem:

\noindent {\bf RH problem 1.}  Find a matrix function $ M(x,t,\lambda)$ satisfying

(i) $ M(x,t,\lambda)$ is analytic in $ \mathbb{C}\backslash\Sigma $;

(ii) The boundary value $ M_{\pm}(x,t,\lambda) $ at $ \Sigma $ satisfies the jump condition
\begin{equation}
 M_+ (x,t,\lambda)  =  M_-(x,t,\lambda)
V(x,t,\lambda),\quad \lambda \in \Sigma, \label{muE}
\end{equation}
and the jump matrix $ V(\lambda) $ is given  by
\begin{equation}
V(x,t,\lambda)=\begin{pmatrix}
1-r(\lambda)\overline{r(\bar{\lambda})}&-\overline{r(\bar{\lambda})}e^{2it\theta(\lambda)} \\
r(\lambda)e^{-2it\theta(\lambda)} &1
\end{pmatrix},\ \ r(\lambda)=\frac{s_{21}(\lambda)}{s_{11}(\lambda)},
\end{equation}
where
\begin{equation}
\theta(x,t,\lambda)=2\lambda^{2}(a\lambda-2b)^{2}+\frac{x}{t}\lambda(a\lambda-2b); \label{ewew}
\end{equation}

(iii)  Asymptotic behavior
\begin{equation}
M(x,t,\lambda)=I+O( \lambda^{-1} ),\qquad as\quad  \lambda\rightarrow\infty.
\end{equation}
The  solution for the initial-value problem (\ref{MNLS}) can be expressed in terms of the  RH problem
\begin{align}
q(x,t)=2ie^{ia\int_{-\infty}^x |q(x',t)|^2dx'}\lim_{\lambda\rightarrow\infty }[\lambda M(x,t,\lambda)]_{12}.\label{q1}
\end{align}

We set
\begin{equation}
m(x,t)=\lim_{\lambda\rightarrow\infty }[\lambda M(x,t,\lambda)]_{12},\label{4.7}
\end{equation}
then from  (\ref{q1}) and its complex conjugate we obtain
\begin{equation}
|q|^{2}=4|m|^2.
\end{equation}
Thus (\ref{q1}) becomes
\begin{equation}
q(x,t)=2ie^{4ia\int_{-\infty}^x|m|^2dx'}m(x,t).\label{4.9}
\end{equation}

Let $z=\lambda(a\lambda-2b)$ and define
\begin{equation}
N(x,t,z)=(a\lambda-2b)^{-\frac{\widehat{\sigma}_{3}}{2}}M(x,t,\lambda),  \label{4.10}
\end{equation}
then we translate the RH problem 1  to  a new RH  problem:

\noindent {\bf RH problem 2.}  Find a matrix function $ N(x,t,z)$ satisfying

(i) $ N(x,t,z)$ is analytic in $ \mathbb{C}\backslash \mathbb{R} $,

(ii) The boundary value $ N_{\pm}(x,t,z) $ at $ \Sigma $ satisfies the jump condition
\begin{equation}
N_+(x,t,z)=N_-(x,t,z)V_{N},\quad z\in \mathbb{R} \label{muE},
\end{equation}
where
$$
V_{N}=e^{-it\theta\widehat{\sigma}_{3}}\begin{pmatrix}
1-z\rho_{1}(z)\rho_{2}(z)  &-\rho_{1}(z)\\
z\rho_{2}(z)&1
\end{pmatrix}.
$$

\begin{equation}
\theta (z)=\frac{x}{t}z+2z^{2},\label{theta}
\end{equation}
and two  reflection  coefficients   are   given by
\begin{align}
\rho_{1}(z)=\frac{\overline{r(\bar{\lambda})}}{a\lambda-2b}, \ \ \rho_{2}(z)=\frac{r(\lambda)}{\lambda}.\label{reflection}
\end{align}

(iii) Asymptotic behavior
\begin{equation}
N(x,t,z)=I+O(\frac{1}{z}),\quad as\quad  z\rightarrow\infty.
\end{equation}
From   (\ref{theta}),  we get the  stationary point
\begin{equation}
z_{0}=-\frac{x}{4t},
\end{equation}
and  two  steepest descent lines
$$
L=\{z=z_{0}+ue^{i\pi/4},\quad u\geq0\}\cup\{z=z_{0}+ue^{5i\pi/4},\quad u\geq0\},
$$
$$
\bar{L}=\{z=z_{0}+ue^{-i\pi/4},\quad u\geq0\}\cup\{z=z_{0}+ue^{3i\pi/4},\quad u\geq0\}.
$$
\textbf{Remark 4.1.}  Comparing with the RH problem 1, the RH problem 2  possesses three
special properties:  1)  It is a RH problem  on  a  real axis; 2) it possesses two reflection
coefficients $\rho_1$ and $\rho_2$;  3) Its phase factor $ \theta(z)  $ has only one stationary point $z_0$.

Next based on the RH problem 2,    we   analyze the asymptotic behavior of the solution of the defocusing mixed NLS equation by
 using the  $\bar\partial$ steepest descent method.

\section{Triangular factorizations of jump matrix}
\hspace*{\parindent}
We decompose  the jump matrix $V_{N}$  into   appropriate  upper/lower triangular factorizations  which can help us to  make   continuous extension
of  the RH problem. It can be shown that the matrix $V_{N}$ admits the following two  triangular factorizations:
\begin{equation}\nonumber
V_N( z)=\begin{cases}
 \begin{pmatrix}
1&-\rho_1(z)e^{-2it\theta}\\
0&1
\end{pmatrix}
\begin{pmatrix}
1&0\\
z\rho_2(z)e^{2it\theta}&1
\end{pmatrix}\triangleq W_LW_R,\quad z>z_0\\ \\
\begin{pmatrix}
1&0\\
\frac{z\rho_{2}(z)}{1-z\rho_1(z)\rho_2(z)}e^{2it\theta}&1
\end{pmatrix}
\begin{pmatrix}
1-z\rho_1(z)\rho_2(z)&0\\
0&\frac{1}{1-z\rho_1(z)\rho_2(z)}
\end{pmatrix}
\begin{pmatrix}
1&\frac{-\rho_1(z)}{1-z\rho_1(z)\rho_2(z)}e^{-2it\theta}\\
0&1
\end{pmatrix}\\ \\
\triangleq U_LU_0U_R, \quad z<z_0.
\end{cases}
\end{equation}
which is shown in figure \ref{fig2}.
\begin{figure}[H]
\begin{center}
\begin{tikzpicture}
\draw[thick,  -> ] (-3,0)--(3,0);
\draw [fill] (0,0) circle [radius=0.03];
\node  [below]  at (0,0) {$z_0$};
\node  [below]  at (3.2,0.2) {$\mathbb{R}$};
\node [thick ] [below]  at (1.5,0.6) {\footnotesize $ V_N=W_LW_R$};
\node [thick ] [below]  at (-1.5,0.6) {\footnotesize $ V_N=U_LU_0U_R$};
\end{tikzpicture}
\end{center}
\caption{The jump matrices of $N$.}
\label{fig2}
\end{figure}
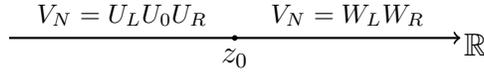

To remove the diagonal matrix $U_0$ across $(-\infty,z_0]$ in the second factorization, we introduce  a scalar RH problem
\begin{equation}
\begin{cases}
\delta_{+}(z)=\delta_{-}(z)(1-z\rho_{1}(z)\rho_{2}(z)),& z<z_{0},\\
\delta_{+}(z)=\delta_{-}(z), & z>z_{0},\\
\delta(z)\rightarrow 1,& as\ z\rightarrow\infty, \label{mp}
\end{cases}
\end{equation}
which  has a solution
\begin{equation}
\begin{aligned}
\delta(z)&=\exp\left[\frac{1}{2\pi i}\int_{-\infty}^{z_{0}}\frac{\mathrm{ln}(1-\xi\rho_{1}(\xi)\rho_{2}(\xi))}{\xi-z}d\xi\right]=\exp\left( i\int_{-\infty}^{z_{0}}\frac{ \nu(\xi)}{\xi-z}d\xi\right),\label{dz}
\end{aligned}
\end{equation}
where
$$\nu(z)=-\frac{1}{2\pi } \mathrm{ln}[1-z\rho_{1}(z)\rho_{2}(z)].$$

We write  (\ref{dz})   in the form
\begin{equation}
\begin{aligned}
\delta(z)&=\exp \left(i \int_{z_0-1}^{z_0} \frac{\nu(z_0)}{\xi-z} \mathrm{d}\xi\right)\exp \left(i \int_{-\infty}^{z_0} \frac{\nu(\xi)-\chi(\xi)\nu(z_0)}{\xi-z}  \mathrm{d}\xi\right)\\
&=(z-z_0)^{i \nu(z_0)} e^{i \beta(z ,z_0)},
\end{aligned}
\end{equation}
with
\begin{equation}
\beta(z,z_0)=-\nu(z_0) \log (z-z_0+1)+\int_{-\infty}^{z_0} \frac{\nu(\xi)-\chi(\xi)\nu(z_0)}{\xi-z} \mathrm{d}\xi,
\end{equation}
where $\chi$ is the characteristic function of the interval $(z_0-1, z_0) .$ We choose the branch of the logarithm with $-\pi<\arg (z)<\pi$ as $z \rightarrow z_0$ for $z-z_0=r e^{i\phi}$ with $-\pi<\phi<\pi$ and implied constants independent of $z_0 \in \mathbb{R}$.

By using  (\ref{omega}) and (\ref{dcx1}), we get
$$ |r(\lambda)|^2+1/|s_{11}(\lambda)|^2=1,$$
which implies that   $ |z\rho_1(z)\rho_2(z)|=|r(\lambda)|<1$. In a similar way \cite{lta29}, we can  show that

\begin{thm}  Suppose that $r(\lambda)=z\rho_1(z)\rho_2(z)\in L^2(\mathbb{R})\cap L^{\infty}(\mathbb{R})$ and $\|r\|_{L^{\infty}} \leq c<1$, the function $ \delta(z) $ has properties:

 (i)  \ $ \delta(z)  $ is  uniformly bounded in $ z $,   namely
\begin{equation}
 (1-||r||_{L^\infty})^{1/2} \leq|\delta(z)| \leq  (1-||r||_{L^\infty})^{-1/2}.
\end{equation}

(ii)  \  $ \delta(z)\overline{\delta(\bar{z})}=1.$

(iii)\   $ \delta(z)$  admits  asymptotic expansion
\begin{equation}
\delta(z)=1+\frac{i}{z} \int_{-\infty}^{z_0} \nu(\xi)\mathrm{d}\xi+\mathcal{O}\left(\frac{1}{z^{2}}\right), \ z \rightarrow \infty.
\end{equation}

(iv) \  Asymptotics  along any ray of the form $z=z_0+e^{i\phi} \mathbb{R}^{+}$ with $-\pi <\phi<\pi$   as $z \rightarrow z_0$
\begin{equation}
\left|\delta(z)-e^{i\beta(z_0,z_0)}(z_0-z)^{-i \nu(z_0)}\right| \lesssim-|z-z_0| \log |z-z_0|.\label{bds}
\end{equation}

\end{thm}

Making  a transformation
\begin{equation}
N^{(1)}=N\delta^{-\sigma_3},\label{delta}
\end{equation}
we get  the following RH problem.\\
\textbf{RH problem 3.}
Find an analytic function $N^{(1)}$  with the following properties:
$N^{(1)}$ satisfies the new Riemann-Hilbert problem:

(i)\quad $ N^{(1)}(x,t,z)$ is analytic in $ \mathbb{C}\backslash \mathbb{R} $;

(ii)\quad The boundary value $ N^{(1)}(x,t,z)$ at $\mathbb{R} $ satisfies the jump condition
\begin{equation}
N^{(1)}_+(x,t,z) =N^{(1)}_-(x,t,z)
V_{N}^{(1)}(x,t,z),\quad z\in \mathbb{R}, \label{mup}
\end{equation}
where the jump matrix $ V_{N}^{(1)}(z) $ is given   by
\begin{equation}
V_{N}^{(1)}=\begin{cases}
\begin{pmatrix}
1&-\rho_{1}(z)\delta_{+}^{2}e^{-2it\theta}\\
0&1
\end{pmatrix}
\begin{pmatrix}
1&0\\
z\rho_{2}(z)\delta_{-}^{-2}e^{2it\theta}&1
\end{pmatrix}\triangleq G_LG_R,\qquad\quad\quad z>z_{0},\\[0.5cm]
\begin{pmatrix}
1&0\\
\frac{z\rho_{2}(z)}{1-z\rho_{1}(z)\rho_{2}(z)}\delta_{-}^{-2}e^{2it\theta}&1
\end{pmatrix}
\begin{pmatrix}
1&\frac{-\rho_{1}(z)}{1-z\rho_{1}(z)\rho_{2}(z)}\delta_{+}^{2}e^{-2it\theta}\\
0&1
\end{pmatrix}\triangleq  H_LH_R,\qquad z<z_{0},
\end{cases}\label{VV}
\end{equation}
see Figure \ref{fig3}.

(iii)  Asymptotic condition
\begin{equation}
N^{(1)}(x,t,z)=I+O(z^{-1}),\qquad as\quad  z\rightarrow\infty.
\end{equation}

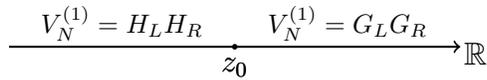
\begin{figure}[H]
\begin{center}
\begin{tikzpicture}
\draw[thick,  -> ] (-3,0)--(3,0);
\node  [below]  at (0,0) {$z_0$};
\draw [fill] (0,0) circle [radius=0.03];
\node  [below]  at (0,0) {$z_0$};
\node  [below]  at (3.2,0.2) {$\mathbb{R}$};
\node [thick ] [below]  at (1.5,0.7) {\footnotesize $ V^{(1)}_N=G_LG_R$};
\node [thick ] [below]  at (-1.5,0.7) {\footnotesize $ V^{(1)}_N=H_LH_R$};
\end{tikzpicture}
\end{center}
\caption{The jump matrix of $N^{(1)}(z)$.}
\label{fig3}
\end{figure}

\section{ A  hybrid $\overline{\partial}$ problem  }\label{sec6}
\hspace*{\parindent}
We continuously  extend the scattering data    in the jump matrix $V_{N}^{(1)}$ by the following way:

(1)   $z\rho_2(z)$  is extended to  the sector $\Lambda_1: \{z:\arg z\in(0,\pi/4)\}$,

(2)  $-\frac{\rho_1(z)}{1-z\rho_1(z)\rho_2(z)}$ is extended to  the sector $\Lambda_3: \{z:\arg z\in(3\pi/4,\pi)\}$,

(3)   $\frac{z\rho_2(z)}{1-z\rho_1(z)\rho_2(z)}$ is extended to  the sector $\Lambda_4: \{z:\arg z\in(\pi,5\pi/4)\}$,

(4)   $-\rho_1(z)$ is extended to   the sector $\Lambda_6: \{z:\arg z\in(7\pi/4,2\pi)\}$.

\begin{lemma}
   There exist function  $R_j$: $\Lambda_j\rightarrow \mathbb{C}$, $j=1,3,4,6$  with boundary values
\begin{align}
&R_1(z)=\begin{cases}
z\rho_2(z)\delta(z)^{-2},& z\in (z_0, \infty),\\[5pt]
z_0\rho_2(z_0)\delta_0(z_0)^{-2}(z-z_0)^{-2i\nu},&  z\in \Sigma_1,
\end{cases}\label{R1}\\
&R_3(z)=\begin{cases}
-\delta_+^2\frac{\rho_1(z)}{1-z\rho_1(z)\rho_2(z)},& z\in (-\infty, z_0),\\[5pt]
-\frac{\rho_1(z_0)}{1-z_0\rho_1(z_0)\rho_2(z_0)}\delta_0(z_0)^2(z-z_0)^{2i\nu},& z\in \Sigma_2,
\end{cases}\label{R3}\\
&R_4(z)=\begin{cases}
\delta_-^{-2}\frac{z\rho_2(z)}{1-z\rho_1(z)\rho_2(z)},& z\in (-\infty, z_0),\\[5pt]
\frac{z_0\rho_2(z_0)}{1-z_0\rho_1(z_0)\rho_2(z_0)}\delta_0(z_0)^{-2}(z-z_0)^{-2i\nu},& z\in \Sigma_3,
\end{cases}\label{R4}\\
&R_6(z)=\begin{cases}
-\rho_1(z)\delta(z)^2,& z\in (z_0,\infty),\\[5pt]
-\rho_1(z_0)\delta_0(z_0)^{2}(z-z_0)^{2i\nu},& z\in \Sigma_4,
\end{cases} \label{R6}
\end{align}
moreover, $R_j$ admit the following estimates
\begin{align}
&|\bar{\partial}R_j|\leq c_1|z-z_0|^{-\frac{1}{2}}+c_2|p'_j(Re z)|, \quad j=1,3,4,6,\nonumber\\
& \bar{\partial}R_j =0, \ j=2, 5.\nonumber
\end{align}
where  
 \begin{align}
&p_1(z)=z\rho_2(z), \ p_3(z)=-\frac{\rho_1(z)}{1-z\rho_1(z)\rho_2(z)},\nonumber\\
& p_4(z)=\frac{z\rho_2(z)}{1-z\rho_1(z)\rho_2(z)}, \ p_6(z)=-\rho_1(z).\nonumber
\end{align}

\end{lemma}
\begin{proof}
We only give the proof for $R_1$, the  others  can be proved in a similar  way. Define $f_{1}(z)$ on $\Lambda_{1}$ by
\begin{equation}
f_{1}(z)=p_{1}(z_0) \delta_{0}(z_0)^{-2} (z-z_0)^{-2i\nu(z_0)} \delta(z)^{2},\label{f1}
\end{equation}
Denote $z=u+iv=z_0+\varrho e^{i\varphi} $, then we have
\begin{equation}
\varrho =|z-z_0|, \ \ \bar\partial =\frac{1}{2}e^{i\varphi}( \partial_{\varrho}+i\varrho^{-1} \partial_{\varphi}).\label{popp}
\end{equation}
Define
\begin{equation}
R_1(u,v)= p_1(u) \cos2\varphi+(1-\cos2\varphi)f_1(u+iv),
\end{equation}
it follows  from  (\ref{R1})   that
\begin{equation}
\bar{\partial}R_1=(p_1-f_1) \bar{\partial} \cos2\varphi +\frac{1}{2}\cos2\varphi  p'_1(u),
\end{equation}
further by using (\ref{popp}), we obtain
\begin{equation}
|\bar{\partial}R_1|\leq\frac{c_1}{|z-z_0|}[|p_1-p_1(z_0)|+|p_1(z_0)-f_1|]+c_2|p'_1(u)|.
\end{equation}
While
\begin{equation}
|p_1(z)-p_1(z_0)|=|\int_{z_0}^zp_1'(s) {ds}|\leq  ||p_1||_{L^2((z,z_0))}\cdot|z-z_0|^{1/2}.
\end{equation}
From (\ref{f1}),  we have
\begin{align}
p_1(z_{0})-f_{1}=&p_1(z_{0})-p_1(z_{0}) \exp [2i\nu((z-z_{0}) \ln (z-z_{0})-(z-z_{0}+1) \ln (z-z_{0}+1))]\nonumber\\
&\times  \exp [2(\beta(z, z_{0})-\beta(z_{0}, z_{0}))].\nonumber
\end{align}
Noticing that in  $\Lambda_1$,
\begin{equation}
|\beta(z,z_0)-\beta(z_0,z_0)|=O(\sqrt{z-z_0}),\nonumber
\end{equation}
and
\begin{equation}
|\nu(z_0)\ln(z-z_0)|\leq O(\sqrt{z-z_0}).\nonumber
\end{equation}
Therefore
\begin{equation}
\begin{aligned}
|p_1(z_0)-f_1|&=p_1(z_0)\{1-\exp [O(\sqrt{z-z_0})]\}\\
&=p_1(z_0)\{O(\sqrt{z-z_0})\}.
\end{aligned}
\end{equation}
Combining these estimates yields
\begin{equation}
|\bar{\partial}R_1|\leq c_1|z-z_0|^{-\frac{1}{2}}+c_2|p'_1|.\label{dbarR}
\end{equation}
\end{proof}

We use $R_j$  obtained above  to define a new unknown function
\begin{equation}
N^{(2)}=N^{(1)}(z)\mathcal{R}^{(2)}(z) \label{RT}
\end{equation}
where
\begin{equation}
\mathcal{R}^{(2)}=\begin{cases}
\begin{pmatrix} 1&0\\ R_1e^{2it\theta} &1\end{pmatrix},& z\in\Lambda_1,\\
\begin{pmatrix} 1&-R_3e^{-2it\theta}\\
0&1\end{pmatrix},& z\in\Lambda_3,\\
\begin{pmatrix} 1&0\\ R_4e^{2it\theta}&1\end{pmatrix},& z\in\Lambda_4,\\
\begin{pmatrix} 1&-R_6e^{-2it\theta}\\ 0&1\end{pmatrix},& z\in\Lambda_6,\\
\begin{pmatrix} 1&0\\0&1\end{pmatrix},& z\in\Lambda_2\cup\Lambda_5,
\end{cases}
\end{equation}
which is shown in  Figure \ref{v-ninfty6}.
\begin{figure}[H]
\begin{center}
\begin{tikzpicture}
\draw [thick ](-2.3,0)--(2.4,0);
\draw [thick, -> ] (-2,0)--(-1,0);
\draw[thick,  -> ] (0,0)--(1,0);
\draw[thick ] (-2,-2)--(2,2);
\draw[thick,-> ](-1.5,-1.5)--(-0.8,-0.8);
\draw[thick,-> ](0,0)--(0.8,0.8);
\draw[thick ](-2,2)--(2,-2);
\draw[thick,-> ](-1.5,1.5)--(-0.8,0.8);
\draw[thick,-> ](0,0)--(0.8,-0.8);
\draw[thick ](-1.5,-1.5)--(1.5,1.5);
\draw[thick ](-1.5,1.5)--(1.5,-1.5);
\coordinate (C) at (-0.6,0);
\fill (C) circle (0pt) node[above] {\footnotesize$\Lambda_3$};
\coordinate (E) at (0,0.4);
\fill (E) circle (0pt) node[above] {\footnotesize$\Lambda_2$};
\coordinate (D) at (0.6,0);
\fill (D) circle (0pt) node[above] {\footnotesize$\Lambda_1$};
\coordinate (F) at (0.6,0);
\fill (F) circle (0pt) node[below] {\footnotesize$\Lambda_6$};
\coordinate (J) at (0,-0.4);
\fill (J) circle (0pt) node[below] {\footnotesize$\Lambda_5$};
\coordinate (k) at (-0.6,0.1);
\fill (k) circle (0pt) node[below] {\footnotesize$\Lambda_4$};
\node  [below]  at (0,-0.1) {$z_0$};
\node [thick ] [below]  at (3, 1) {\footnotesize $  \mathcal{R}^{(2)}=\left(\begin{array}{cc} 1&0\\ R_1e^{2it\theta} &1\end{array}  \right) $};
\node [thick ] [below]  at (3.3,-0.4) {\footnotesize $ \mathcal{R}^{(2)}=\left(\begin{array}{cc} 1&-R_6e^{-2it\theta}\\ 0&1\end{array}  \right) $};
\node [thick ] [below]  at (-3.1, 1.1) {\footnotesize $  \mathcal{R}^{(2)}=\left(\begin{array}{cc} 1&-R_3e^{-2it\theta}\\
0&1\end{array}  \right) $};
\node [thick ] [below]  at (-3.2,-0.5) {\footnotesize $ \mathcal{R}^{(2)}=\left(\begin{array}{cc} 1&0\\ R_4e^{2it\theta}&1\end{array}  \right) $};
\node [thick ] [below]  at (0,2) {\footnotesize $  \mathcal{R}^{(2)}=I $};
\node [thick ] [below]  at (0,-1) {\footnotesize $ \mathcal{R}^{(5)}=I $};
\end{tikzpicture}
\end{center}
\caption{ $\mathcal{R}^{(2)}$ in each $\Lambda_j$.}
\label{v-ninfty6}
\end{figure}
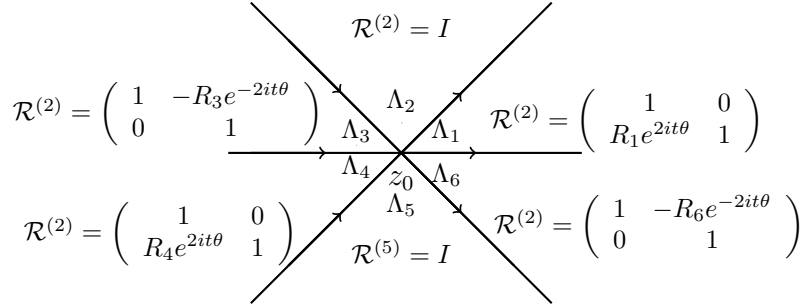

Making $\bar\partial$-differential on the equation (\ref{RT})  and noting that $N^{(2)}$   is analytic in $\Lambda_j, \ j=1, 3, 4, 6$,  we have
\begin{equation}
\overline{\partial} N^{(2)} = N^{(2)}(R^{(2)})^{-1}\overline{\partial} \mathcal{R}^{(2)}= N^{(2)}\overline{\partial} \mathcal{R}^{(2)}.
\end{equation}
 Let $\Sigma^{(2)}=\cup_{j=1}^4 \Sigma_j^{(2)}$.   It  can be shown that $N^{(2)}$ satisfies the following $\bar{\partial}$-RH problem.\\
\textbf{ RH problem 4.} Find a function $N^{(2)}$  with the following properties.

(i)\quad $ N^{(2)}(x,t,z)$ is continuous  in $ \mathbb{C}\setminus \Sigma^{(2)} $;

(ii)\quad The boundary value $ N^{(2)}(x,t,z)$ at $ \Sigma^{(2)}$ satisfies the jump condition
\begin{equation}
N^{(2)}_+ (x,t,z) =N^{(2)}_- (x,t,z)
V_{N}^{(2)}(x,t,z),\quad z\in \Sigma^{(2)}, \label{mup}
\end{equation}
where the jump matrix $ V_{N}^{(2)}(z) $ is defined  by (see Figure \ref{v-ninfty2})
\begin{equation} \label{vnn}
V_{N}^{(2)}=\begin{cases}
e^{-it\theta\widehat{\sigma}_{3}}\begin{pmatrix}
1&0\\
-z_0\rho_2(z_0)\delta_0^{-2}(z-z_0)^{-2i\nu}&1
\end{pmatrix},& z\in\Sigma^{(2)}_1,\\
e^{-it\theta\widehat{\sigma}_{3}}\begin{pmatrix}
1&\frac{\rho_1(z_0)}{1-z_0\rho_1(z_0)\rho_2(z_0)}\delta_0^2(z-z_0)^{2i\nu}\\
0&1
\end{pmatrix},& z\in\Sigma^{(2)}_2,\\
e^{-it\theta\widehat{\sigma}_{3}}\begin{pmatrix}
1&0\\
\frac{z_0\rho_2(z_0)}{1-z_0\rho_1(z_0)\rho_2(z_0)}\delta_0^{-2}(z-z_0)^{-2i\nu}&1
\end{pmatrix},& z\in\Sigma^{(2)}_3,\\
e^{-it\theta\widehat{\sigma}_{3}}\begin{pmatrix}
1&-\rho_1(z_0)\delta_0^2(z-z_0)^{2i\nu}\\
0&1
\end{pmatrix},& z\in\Sigma^{(2)}_4,
\end{cases}
\end{equation}

(iii)  Asymptotic condition
\begin{equation}
N^{(2)}(x,t,z)=I+O(z^{-1}),\qquad as\quad  z\rightarrow\infty.
\end{equation}

 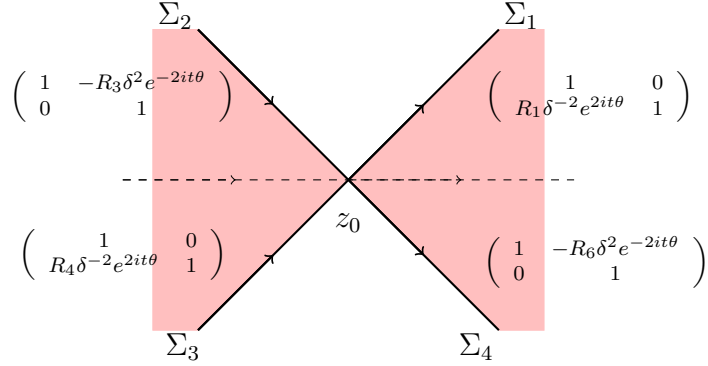
\begin{figure}
\begin{center}
\begin{tikzpicture}
\draw[pink, fill=pink] (0,0)--(-2,-2)--(-2.6,-2)--(-2.6, 2)--(-2, 2 )--(0,0);
\draw[pink, fill=pink] (0,0)--( 2,-2)--( 2.6,-2)--( 2.6, 2)--( 2, 2 )--(0,0);
\draw[dashed] (-3,0)--(3,0);
\draw[dashed]  [ -> ](-3,0)--(-1.5,0);
\draw[dashed]  [   -> ](0,0)--(1.5,0);
\draw[thick ](-2,-2)--(2,2);
\draw[thick,-> ](-2,-2)--(-1,-1);
\draw[thick, -> ](0,0)--(1,1);
\draw[thick ](-2,2)--(2,-2);
\draw[thick,-> ](-2,2)--(-1,1);
\draw[thick, -> ](0,0)--(1,-1);
\node  [below]  at (2.3,2.5 ) {$\Sigma_1$};
\node  [below]  at (-2.3,2.5 ) {$\Sigma_2$};
\node  [below]  at (-2.2 ,-1.9) {$\Sigma_3$};
\node  [below]  at (1.7,-1.9) {$\Sigma_4$};
\node  [below]  at (0,-0.3) {$z_0$};

\node [thick] [below]  at (3.2, 1.6) {\scriptsize $ \left(\begin{array}{cc} 1&0\\ R_1 \delta^{-2} e^{2it\theta} &1\end{array}  \right)  $};
\node [thick] [below]  at (-3,-0.5) {\scriptsize $ \left(\begin{array}{cc} 1&0\\ R_4\delta^{-2}e^{2it\theta} &1\end{array}  \right) $};
\node [thick] [below]  at (-3,1.6) {\scriptsize $\left(\begin{array}{cc} 1& -R_3\delta^{2}e^{-2it\theta} \\ 0&1\end{array}  \right)$};
\node [thick] [below]  at (3.3,-0.6) {\scriptsize $\left(\begin{array}{cc} 1& -R_6\delta^{2}e^{-2it\theta}  \\ 0&1\end{array}  \right)$};
\end{tikzpicture}
\end{center}
\caption{ The jump matrices $V_{N}^{(2)}$ for $N^{(2)}$.  $\bar\partial R^{(2)}\not=0$ in pink domian; and $\bar\partial R^{(2)} =0$ in white domian  }
\label{v-ninfty2}
\end{figure}

(iv) Away from $\Sigma^{(2)}$,  we have
\begin{equation}
\overline{\partial} N^{(2)}= N^{(2)} \overline{\partial} \mathcal{R}^{(2)}, \ \ z\in \mathbb{C}\backslash\Sigma^{(2)},\label{6.22}
\end{equation}
 where
\begin{equation}
\overline{\partial}\mathcal{R}^{(2)}=\begin{cases}
\begin{pmatrix}
0&0\\
\overline{\partial}R_1e^{2it\theta}&1
\end{pmatrix},& z\in\Lambda_1,\\
\begin{pmatrix}
0&-\overline{\partial}R_3e^{-2it\theta}\\
0&0
\end{pmatrix},& z\in\Lambda_3,\\
\begin{pmatrix}
0&0\\
\overline{\partial}R_4e^{2it\theta}&0
\end{pmatrix},& z\in\Lambda_4,\\
\begin{pmatrix}
0&-\overline{\partial}R_6e^{-2it\theta}\\
0&0
\end{pmatrix},& z\in\Lambda_6,\\
\begin{pmatrix}
0&0\\
0&0
\end{pmatrix},& \mathrm{otherwise}.
\end{cases}
\end{equation}
Diagrammatically, the jump matrices are as in Figure \ref{v-ninfty2}.

In order to solve the RH problem 4,  we decompose it into a solvable  RH Problem for $N^{rhp} (x,t,z)$
with $\overline{\partial}R^{(2)}=0$ and a pure $\overline{\partial}$-Problem  $E(x,t,z)$  with $\overline{\partial}R^{(2)} \neq 0$.
The pure RH problem for  $N^{rhp}(z)$  is a solvable model
 associated with   a  Weber equation, which will be analyzed in next section 6;
      The error estimates on the  pure $\overline{\partial}$-problem  for   $E(x,t,z)$ will   given  in  In section 7.

\section{Analysis on a  solvable model}

The  hybrid RH problem  $N^{(2)}(x,t,z)$ with  $\overline{\partial}R^{(2)}=0$ leads to  the following pure  RH problem for the $M^{rhp} (x,t,z)$.\\

\noindent\textbf{ RH problem 5.}  Find a matrix-valued function $N^{rhp} (x,t,z)$ with following properties:

(i) Analyticity: $N^{rhp} (z)$ is analytical in $\mathbb{C} \backslash  \Sigma^{(2)}  $,

(ii) Jump condition:
\begin{equation}
N_{+}^{rhp}( z)=N_{-}^{rhp}( z) V_{N}^{(2)}( z), \quad z \in  \Sigma^{(2)},
\end{equation}
where $ V_{N}^{(2)}( z)$   is given by  (\ref{vnn}).

(iii)   Asymptotic condition
\begin{equation}
N^{rhp}(z)=I+O(z^{-1}),\qquad as\quad  z\rightarrow\infty.
\end{equation}

 Making  a  scaling   transformation 
\begin{align}
 & N^{sol}(\zeta)= N^{rhp}(z)\big|_{ z=\zeta/\sqrt{8t}+z_0},\label{t1}\\
&\rho_2(z_0)=\rho_{20}\delta_0^2e^{-2i\nu(z_0)\ln\sqrt{8t}}e^{4itz_0^2},\label{t2}\\
&\rho_1(z_0)=\rho_{10}\delta_0^{-2}e^{2i\nu(z_0)\ln\sqrt{8t}}e^{-4itz_0^2},\label{t3}
\end{align}
then the RH problem 5 is changed into the  following  RH problem. \\

\noindent\textbf{ RH problem 6.} Find a   matrix-valued function $N^{sol}(\zeta)$  with the following properties:

(i)\quad $N^{sol}(\zeta)$ is analytic on $\zeta\in\mathbb{C}\backslash \Sigma^{(2)}$,

(ii)\quad The boundary value $N^{sol}(\zeta)$  satisfies the jump condition
\begin{equation}
N_+^{sol}(\zeta)=N_-^{sol}(\zeta) V_{N}^{(2)}(\zeta),\quad \zeta\in \Sigma^{(2)},
\end{equation}
where
\begin{equation}
 V_{N}^{(2)}(\zeta)=\begin{cases}
\begin{pmatrix}
1&0\\
z_0\rho_{20}\zeta^{-2i\nu}e^{i\frac{\zeta^2}{2}}&1
\end{pmatrix},\ \qquad \zeta\in \Sigma_1,\\
\begin{pmatrix}
1&-\frac{\rho_{10}}{1-z_0\rho_{10}\rho_{20}}\zeta^{2i\nu}e^{-i\frac{\zeta^2}{2}}\\
0&1
\end{pmatrix},\quad \zeta\in \Sigma_2,\\
\begin{pmatrix}
1&0\\
\frac{z_0\rho_{20}}{1-z_0\rho_{10}\rho_{20}}\zeta^{-2i\nu}e^{i\frac{\zeta^2}{2}}&1
\end{pmatrix},\quad \zeta\in \Sigma_3,\\
\begin{pmatrix}
1&-\rho_{10}\zeta^{2i\nu}e^{-i\frac{\zeta^2}{2}}\\
0&1
\end{pmatrix},\ \qquad \zeta\in \Sigma_4.
\end{cases}\label{sol}
\end{equation}

(iii) Asymptotic condition
\begin{equation}
N^{sol}(\zeta)=I+\frac{N_1^{sol}(\zeta)}{\zeta}+\mathcal{O}( \zeta^{-1}) \quad as\  \zeta\rightarrow\infty ,
\end{equation}

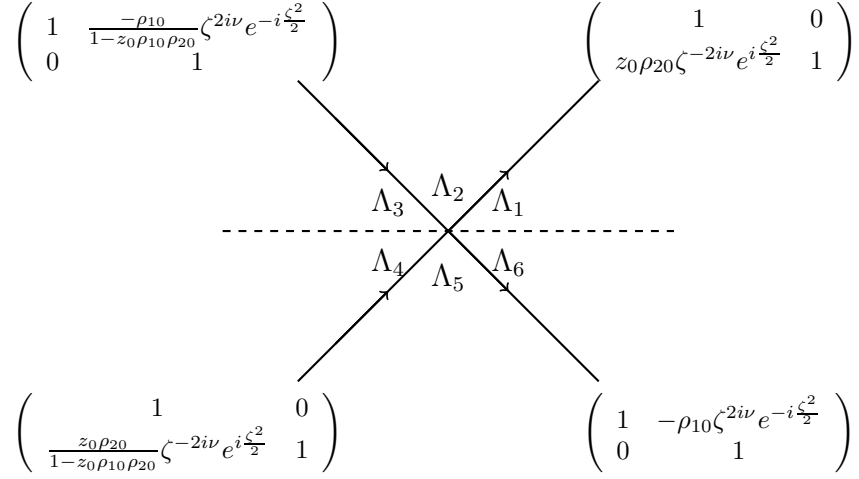
\begin{figure}[H]
\begin{center}
\begin{tikzpicture}
\draw [dashed][thick ](-3,0)--(3,0);
\draw[thick ] (-2,-2)--(2,2);
\draw[thick,-> ](-1.5,-1.5)--(-0.8,-0.8);
\draw [thick,-> ](0,0)--(0.8,0.8);
\draw[thick ](-2,2)--(2,-2);
\draw[thick,-> ](-1.5,1.5)--(-0.8,0.8);
\draw[thick,-> ](0,0)--(0.8,-0.8);
\node  [below]  at (0.8,0.7) {$ \Lambda_1$};
\node  [below]  at (0, 0.9) {$ \Lambda_2$};
\node  [below]  at (-0.8, 0.7) {$ \Lambda_3$};
\node  [below]  at (-0.8,-0.1) {$ \Lambda_4$};
\node  [below]  at (0, -0.3) {$ \Lambda_5$};
\node  [below]  at (0.8,-0.1) {$ \Lambda_6$};
\node [thick ] [below]  at (3.6, 3.2) {\footnotesize $   \left(\begin{array}{cc} 1&0\\
z_0\rho_{20}\zeta^{-2i\nu}e^{i\frac{\zeta^2}{2}}&1\end{array}  \right) $};
\node [thick ] [below]  at (-3.6,3.2) {\footnotesize $  \left(\begin{array}{cc} 1&\frac{-\rho_{10}}{1-z_0\rho_{10}\rho_{20}}\zeta^{2i\nu}e^{-i\frac{\zeta^2}{2}}\\
0&1\end{array}  \right) $};
\node [thick ] [below]  at (-3.6,-2) {\footnotesize $  \left(\begin{array}{cc} 1&0\\
\frac{z_0\rho_{20}}{1-z_0\rho_{10}\rho_{20}}\zeta^{-2i\nu}e^{i\frac{\zeta^2}{2}}&1\end{array}  \right)$};
\node [thick ] [below]  at (3.6,-2) {\footnotesize $  \left(\begin{array}{cc} 1&-\rho_{10}\zeta^{2i\nu}e^{-i\frac{\zeta^2}{2}}\\
0&1\end{array}  \right)$};
\end{tikzpicture}
\end{center}
\caption{ $ \Sigma^{sol}$ and domains $\Lambda_j$, $j=1,\ldots6$.}
\label{vn5}
\end{figure}

The contour  $\Sigma^{(2)}$ and real axis  $\mathbb{R}$  divide complex plane $\mathbb{C}$ into six  different domains $\Lambda_j$, $j=1, \ldots,6$ which are   shown in Figure \ref{vn5}.
 Let
\begin{equation}
N^{sol}(\zeta)=\vartheta(\zeta)P_0e^{\frac{i}{4}\zeta^2\sigma_3}\zeta^{-i\nu(z_0)\sigma_3},\label{pcd}
\end{equation}
 where
\begin{equation}
P_0=\begin{cases}
\begin{pmatrix}
1&0\\
z_0\rho_{20}&1
\end{pmatrix},& \zeta\in \Lambda_1\\
\begin{pmatrix}
1&-\frac{\rho_{10}}{1-z_0\rho_{10}\rho_{20}}\\
0&1
\end{pmatrix},& \zeta\in \Lambda_3\\
\begin{pmatrix}
1&0\\
\frac{z_0\rho_{20}}{1-z_0\rho_{10}\rho_{20}}&1
\end{pmatrix},& \zeta\in \Lambda_4\\
\begin{pmatrix}
1&-\rho_{10}\\
0&1
\end{pmatrix},& \zeta\in \Lambda_6\\
\begin{pmatrix}
1&0\\
0&1
\end{pmatrix},& \zeta\in \Lambda_2\cup\Lambda_5.
\end{cases}
\end{equation}
It is easy to  check that  $\vartheta$  satisfies the following RH problem, see  \cite{CF}.\\
\textbf{ RH problem 7.}  Find a   matrix-valued function $\vartheta(\zeta)$  with the following properties

(i)\quad $ \vartheta $ is analytic in $ \mathbb{C}\backslash \mathbb{R} $,

(ii)\quad The boundary value $\vartheta(\zeta)$ at $\mathbb{R}$ satisfies the jump condition
\begin{equation}
\vartheta_+(\zeta)=\vartheta_-(\zeta)
V(0),\quad \zeta\in \mathbb{R} , \label{beiyong}
\end{equation}
where
$$
V(0)= \begin{pmatrix}
1-z_0\rho_{1}(z_0)\rho_{2}(z_0)  &-\rho_{1}(z_0)\\
z_0\rho_{2}(z_0)&1
\end{pmatrix}.
$$

(iii) Asymptotic condition
\begin{equation}
\vartheta e^{\frac{i\zeta^2}{4}\sigma_{3}}\zeta^{-i\nu(z_0)\sigma_3}\rightarrow I ,\quad as \ \zeta\rightarrow\infty .
\end{equation}
This kind of RH problem can be changed into a Werber equation to get the solution in terms of parabolic cylinder functions.

\begin{proposition}    The   solution to the RH problem 6 is given by $\vartheta_+(\zeta)$ and $\vartheta_-(\zeta)$, defined in $\mathbb{C}_+$ and $\mathbb{C}_-$ respectively. For $\zeta\in \mathbb{C}_+$,
\begin{subequations}
\begin{align}
&(\vartheta_+)_{11}(\zeta)=e^{-\frac{3\pi\nu}{4}}D_{i\nu}(e^{-\frac{3i\pi}{4}}\zeta),\nonumber\\
&(\vartheta_+)_{12}(\zeta)=e^{\frac{\pi\nu}{4}}(\beta_{21})^{-1}[\partial_\zeta D_{-i\nu}(e^{-\frac{i\pi}{4}}\zeta)-\frac{i\zeta}{2}D_{-i\nu}(e^{-\frac{i\pi}{4}}\zeta)],\nonumber\\
&(\vartheta_+)_{21}(\zeta)=e^{-\frac{3\pi\nu}{4}}(\beta_{12})^{-1}[\partial_\zeta D_{i\nu}(e^{-\frac{3i\pi}{4}}\zeta)+\frac{i\zeta}{2}D_{i\nu}(e^{-\frac{3i\pi}{4}}\zeta)],\nonumber\\
&(\vartheta_+)_{22}(\zeta)=e^{\frac{\pi\nu}{4}}D_{-i\nu}(e^{-\frac{i\pi}{4}}\zeta).\nonumber
\end{align}
\end{subequations}
For $\zeta\in \mathbb{C}_-$,
\begin{subequations}
\begin{align}
&(\vartheta_-)_{11}(\zeta)=e^{\frac{\pi\nu}{4}}D_{i\nu}(e^{\frac{i\pi}{4}}\zeta),\nonumber\\
&(\vartheta_-)_{12}(\zeta)=e^{-\frac{3\pi\nu}{4}}(\beta_{21})^{-1}[\partial_\zeta D_{-i\nu}(e^{-\frac{3i\pi}{4}}\zeta)-\frac{i\zeta}{2}D_{-i\nu}(e^{-\frac{3i\pi}{4}}\zeta)],\nonumber\\
&(\vartheta_-)_{21}(\zeta)=e^{\frac{\pi\nu}{4}}(\beta_{12})^{-1}[\partial_\zeta D_{i\nu}(e^{\frac{i\pi}{4}}\zeta)+\frac{i\zeta}{2}D_{i\nu}(e^{\frac{i\pi}{4}}\zeta)],\nonumber\\
&(\vartheta_-)_{22}(\zeta)=e^{-\frac{3\pi\nu}{4}}D_{-i\nu}(e^{\frac{3i\pi}{4}}\zeta).\nonumber
\end{align}
\end{subequations}
where
\begin{equation}
\beta_{12}=\frac{(2\pi)^{\frac{1}{2}}e^{\frac{i\pi}{4}}
e^{-\frac{\pi\nu}{2}}}{z_0\rho_{20}\Gamma(-i\nu)}, \ \ \  {\beta}_{21}  =  \frac{\nu}{\beta_{12}},  \label{beta}
\end{equation}
and $D_{a}(\xi)=D_a(e^{-\frac{3i\pi}{4}}\zeta)$ is a solution of the Weber equation
\begin{align}
\partial_\xi^2D_{a}(\xi)+ \left[\frac{1}{2}-\frac{\xi^2}{4}+a\right]D_{a}(\xi)=0.  \nonumber
\end{align}
\end{proposition}
Finally, by using the proposition and (\ref{pcd}),  we then  get
\begin{equation}
N^{sol} =I+ \frac{N^{sol}_1}{\zeta}+\mathcal{O}(\zeta^{-2}),\nonumber
\end{equation}
where
\begin{equation}
N^{sol}_1 = \begin{pmatrix}
0&-i\beta_{12}\\
i\beta_{21}&0
\end{pmatrix} .
\end{equation}

\section{ Analysis on a  pure $\bar{\partial}$ problem }
\hspace*{\parindent}
Suppose that $N^{(2)}$ is a solution of the RH problem 4,  then the ratio
\begin{equation}
E( z)=N^{(2)}( z)N^{rhp}(z)^{-1},\label{ET}
\end{equation}
will have no jumps in the plane, but is a solution of the following problem. \\
 \textbf{RH  problem 8.}  Find a function $E( z)$ with the following properties:

(i) $E( z)$ is continuous in $\mathbb{C}$,

(ii) Asymptotic condition
\begin{equation}
E (z)=I+O(z^{-1}),\qquad as\quad  z\rightarrow\infty.
\end{equation}

(iii)   $E(x,t,z)$ satisfies the $\bar{\partial}$-equation
\begin{equation}
\begin{aligned}
\overline{\partial}E=EW( z), \label{dbar}
\end{aligned}
\end{equation}
where
\begin{equation}\label{8.4}
W(z) = \begin{cases}
N^{rhp}(z)\begin{pmatrix}
0&0\\
\overline{\partial}R_1e^{2it\theta}\delta^{-2}&0
\end{pmatrix}N^{rhp}(z)^{-1},&  z\in \Lambda_1,\\ \\
N^{rhp}(z) \begin{pmatrix}
0&-\overline{\partial}R_3e^{-2it\theta}\delta^2\\
0&0
\end{pmatrix}N^{rhp}(z)^{-1},& z\in\Lambda_3,\\ \\
N^{rhp}(z) \begin{pmatrix}
0&0\\
\overline{\partial}R_4e^{2it\theta}\delta^{-2}&0
\end{pmatrix}N^{rhp}(z)^{-1},& z\in\Lambda_4,\\ \\
N^{rhp}(z) \begin{pmatrix}
0&-\overline{\partial}R_6e^{-2it\theta}\delta^2\\
0&0
\end{pmatrix}N^{rhp}(z)^{-1},& z\in\Lambda_6,\\ \\
\begin{pmatrix}
0&0\\
0&0
\end{pmatrix}, & otherwise.
\end{cases}
\end{equation}

\begin{proof}
Noticing that $N^{rhp}(z) $ is holomorphic in $\mathbb{C}\backslash\Sigma^{(2)}$,  by using (\ref{6.22}),
direct calculation shows that

\begin{equation}\nonumber
\begin{aligned}
\overline{\partial}E&=\overline{\partial}N^{(2)}(N^{PC})^{-1} =N^{(2)}\overline{\partial}\mathcal{R}^{(2)}(N^{PC})^{-1}\\
&=\{ N^{(2)}(N^{PC})^{-1}\} \{    N^{PC} \overline{\partial}\mathcal{R}^{(2)}(N^{PC})^{-1}\}=EW(x,t,z),
\end{aligned}
\end{equation}
where $W(x,t,z)$ is given by (\ref{8.4}).
\end{proof}

The solution of the $\bar\partial$-equation (\ref{dbar}) can be expressed by  the following integral
\begin{equation}
E=I-\frac{1}{\pi}\int\int\frac{EW}{s-z}\mathrm{dA(s)},\label{E}
\end{equation}
 which can be written    in operator equation
\begin{equation}
 (I-J) E =I,\label{JE}
\end{equation}
where
\begin{equation}
J(E)=-\frac{1}{\pi}\int\int\frac{EW}{s-z}\mathrm{dA(s)}.\label{JE}
\end{equation}
In order to show the solvability of (\ref{JE}),  we  prove  that $J$ is small    norm as $t\rightarrow \infty$.
\begin{proposition}   As $t\rightarrow \infty$,  we have the following estimate
\begin{equation}
||J||_{L^\infty\rightarrow L^\infty}\leq ct^{-1/4}.
\end{equation}
\end{proposition}

\begin{proof} We give the details for sector $\Lambda_1$ only as the corresponding arguments for other
sectors which are identical with appropriate modifications. Let $f\in L^\infty(\Lambda_1)$, then
\begin{equation}
\begin{aligned}
|J(f)| &\leq\int\int_{\Lambda_1}\frac{|f\bar{\partial}R_1\delta^{-2}e^{2it\theta}|}{|s-z|} {dA(s)}\\
&\leq||f||_{L^\infty(\Lambda_1)}||\delta^{-2}||_{L^\infty(\Lambda_1)}\int\int_{\Lambda_1}\frac{|\bar{\partial}R_1e^{2it\theta}|}{|s-z|} {dA(s)}\\
&\leq  I_1+I_2,
\end{aligned}
\end{equation}
where
\begin{subequations}
\begin{align}
&I_1=\int\int_{\Lambda_1}\frac{|p'_1|e^{-4tv(u-z_0)}}{|s-z|}{dA(s)},\\
&I_2=\int\int_{\Lambda_1}\frac{|s-z_0|^{-1/2}e^{-4tv(u-z_0)}}{|s-z|}{dA(s)}.
\end{align}
\end{subequations}
Since $\rho_1(z),\rho_2(z)\in H^{2,2}(\mathbb{R})$ and we set $s=u+iv$, $z=\alpha+i\tau$,  then
\begin{equation}
\begin{aligned}
I_1&\leq\int_0^\infty\int_v^\infty\frac{|p'_1|e^{-4tv(u-z_0)}}{|s-z|}{dudv}\\
&\leq\int_0^\infty e^{-4tv^2}\int_v^\infty\frac{|p'_1|}{|s-z|}{dudv}\\
&\leq \|p_1^{\prime}(u)\|_{L^{2}(\mathbb{R})}\int_0^\infty e^{-4tv^2}||\frac{1}{s-z}||_{L^2((v,\infty))}{dv}.
\end{aligned}
\end{equation}
Moreover,
\begin{equation}
\begin{aligned}
||\frac{1}{s-z}||_{L^2((v,\infty))}&\leq \left(\int_\mathbb{R}\frac{1}{|s-z|^2}{du}\right)^{1/2} \leq \left(\frac{\pi}{|v-\tau|}\right)^{1/2},
\end{aligned}
\end{equation}
Thus we have
\begin{equation}
|I_1|\leq C_1\int_0^\infty\frac{e^{-4tv^2}}{\sqrt{\tau-v}} {dv}= C_1\left[\int_0^\tau\frac{e^{-4tv^2}}{\sqrt{\tau-v}} {dv}+\int_\tau^\infty\frac{e^{-4tv^2}}{\sqrt{v-\tau}} {dv}\right].
\end{equation}
Using the fact $\sqrt{\tau}e^{-4t\tau^2p^2}\leq ct^{-1/4}p^{-1/2}$, we obtain
\begin{equation}
\begin{aligned}
\int_0^\tau\frac{e^{-4tv^2}}{\sqrt{\tau-v}} {dv}&\leq\int_0^1\sqrt{\tau}\frac{e^{-4t\tau^2p^2}}{\sqrt{1-p}} {dp}
&\leq ct^{-1/4}\int_0^1\frac{1}{\sqrt{p(1-p)}} {dp}\leq c_1t^{-1/4},
\end{aligned}
\end{equation}
whereas using the variable substitution $w=v-\tau$,
\begin{equation}
\int_\tau^\infty\frac{e^{-4tv^2}}{\sqrt{v-\tau}}{dv}\leq\int_0^\infty\frac{e^{-4tw^2}}{\sqrt{w}} {dw}.
\end{equation}
According to $e^{-tw^2}w^{5/2}\leq ct^{-1/4}$, it becomes
\begin{equation}
\int_\tau^\infty\frac{e^{-4tv^2}}{\sqrt{v-\tau}}{dv}\leq c_2t^{-1/4}.
\end{equation}
Hence the final estimation is
\begin{equation}
|I_1|\leq C_2t^{-1/4}.\label{I1}
\end{equation}
To arrive at a similar estimate for $I_2$ , we start with the following $L^p$-estimate for $p>2$.
\begin{equation}
\begin{aligned}
|||s-z_0|^{-1/2}||_{L^p(du)}&=(\int_{z_0+v}^\infty\frac{1}{|u+iv-z_0|^{p/2} }{du})^{1/p}\\
&=(\int_{v}^\infty\frac{1}{|u+iv|^{p/2}} {du})^{1/p}\\
&=(\int_{v}^\infty\frac{1}{(u^2+v^2)^{p/4}} {du})^{1/p}\\
&=v^{(1/p-1/2)}\left(\int_{\pi/4}^{\pi/2}(\cos x)^{p/2-2} {dx}\right)^{1/p}\\
&\leq cv^{(1/p-1/2)}
\end{aligned}
\end{equation}
Similarly to the $L^2$-estimate above, we obtain for $L^q$ with  $1/p+1/q=1$.
\begin{equation}
||\frac{1}{s-z}||_{L^p(v,\infty)}\leq c|v-\tau|^{1/q-1}.
\end{equation}
It follows that
\begin{equation}
|I_2|\leq c[\int_0^\tau e^{-4tv^2}v^{(1/p-1/2)}|v-\tau|^{1/q-1}{dv}+\int_\tau^\infty e^{-tv^2}v^{(1/p-1/2)}|v-\tau|^{1/q-1}{dv}].
\end{equation}
By  using the fact $\sqrt{\tau}e^{-4t\tau^2w^2}\leq ct^{-1/4}w^{-1/2}$,  the first integral yields
\begin{equation}
\int_0^\tau e^{-4tv^2}v^{(1/p-1/2)}|v-\tau|^{1/q-1}{dv}\leq ct^{-1/4}.\label{I1}
\end{equation}
Let $v=\tau+w$,  the estimate for the second integral$I_2$  leads to 
\begin{equation}
\int_0^\infty e^{-4t(\tau+w)^2}(\tau+w)^{(1/p-1/2)}w^{1/q-1}{dw}\leq \int_0^\infty e^{-tw^2}w^{-1/2}{dw}.
\end{equation}
Then making use of the variable substitution $y=\sqrt{t}w$ yields
\begin{equation}
\int_\tau^\infty e^{-tv^2}v^{(1/p-1/2)}|v-\tau|^{1/q-1}{dv}\leq ct^{-1/4}.
\end{equation}
Combining the previous estimates of $I_1$ in (\ref{I1}), the final result is described below
\begin{equation}
|I_2|\leq ct^{-1/4}.\label{I2}
\end{equation}
Finally,  combing (\ref{I1}) and (\ref{I2}) gives 
\begin{equation}
|J |\leq ct^{-1/4}.
\end{equation}
\end{proof}

\begin{proposition}
For sufficiently large $t$,  the integral equation (\ref{JE}) may be inverted by Neumann series. Its asymptotic expression is
\begin{equation}
E(z)=I+O(t^{-1/4}).\label{E}
\end{equation}
\end{proposition}
After proving the existence and analyzing the asymptotic result of $E( z)$, our goal is to establish the relation between $E( z)$ and $N( z)$, and to obtain the asymptotic behavior of $N( z)$ by using the estimation (\ref{E}).

\section{ Long time asymptotics of the mixed NLS  }
\hspace*{\parindent}
Recall transformations (\ref{delta}), (\ref{RT}) and (\ref{ET}), we get
\begin{equation}
\begin{aligned}
N( z)&=E( z)N^{rhp}(z)\mathcal{R}^{(2)}(z)^{-1}\delta^{\sigma_3}.\label{10}
\end{aligned}
\end{equation}
Making the  Laurent expansions
\begin{subequations}
\begin{align}
&N(z)= I+\frac{N_1}{z}+o(z^{-1}),\nonumber\\
&E(z)= I+\frac{E_1}{z}+o(z^{-1}),\nonumber\\
&N^{rhp}(z)=I+\frac{N^{sol}_1}{\sqrt{8t}(z-z_0)} +o(z^{-1}),\label{12} \nonumber\\
&\delta(z)^{\sigma_3}=I+\frac{\delta_1{\sigma_3}}{z} +o(z^{-1}), \nonumber
\end{align}
\end{subequations}
and taking $z\rightarrow \infty$ in vertical direction in $\Lambda_2, \Lambda_5$, the formula (\ref{10})  gives
\begin{equation}
 N_1 =\frac{1 }{\sqrt{8t}}N^{sol}_1 +E_1+\delta_1\sigma_3.\label{13}
\end{equation}

 Next we evaluate the decay rate of the matrix  $E_1$.  From the integral equation (\ref{E}) satisfied by $E$, we have
\begin{equation}
E_1=\frac{1}{\pi}\int\int EW {dA(s)},
\end{equation}
which can be  divided  into two parts for calculation
\begin{equation}\label{94}
\begin{aligned}
|E_1|&\leq\int_0^\infty\int_{v+z_0}^\infty e^{-4tv(u-z_0)}|p'_1| {dudv}+\int_0^\infty\int_{v+z_0}^\infty e^{-4tv(u-z_0)}|s-z_0|^{-1/2}{dudv}\\
&=I_3+I_4.
\end{aligned}
\end{equation}

By using  the Cauchy-Schwarz inequality, we have
\begin{equation}
\begin{aligned}
I_3&\leq c\int_0^\infty \left(\int_v^\infty e^{-4tuv} {du} \right)^{1/2} {dv}\leq ct^{-3/4}.\label{95}
\end{aligned}
\end{equation}
Similarly, by the H\"{o}lder inequality for $1/p+1/q=1$, $2<p<4$,  we let $w=\sqrt{t}v$ and get
\begin{equation}
\begin{aligned}
I_4&  \leq\frac{c}{t^{1/q}}\int_0^\infty v^{2/p-3/2}e^{-4tv^2} {dv}\leq c t^{-3/4} \int_{0}^{\infty} w^{2 / p-3 / 2} e^{-4 w^{2}} d w \leq c t^{-3 / 4}\leq ct^{-3/4}.\label{96}
\end{aligned}
\end{equation}
Substituting (\ref{95}) and (\ref{96}) into (\ref{94}) yields
\begin{equation}
|E_1|\leq ct^{-3/4}.\label{E1}
\end{equation}

Recall the formula (\ref{4.9}), we know
\begin{equation}
q(x,t)=2ie^{4ia\int_{-\infty}^x|m(x',t)|^2dx'}m(x,t),\label{419}
\end{equation}
which implies that  we just make an estimate on the $m(x,t)$ to obtain the asymptotic of the potential  $q(x,t)$.

From (\ref{4.7}), (\ref{4.10}), (\ref{13})  and  (\ref{E1}),  we have
\begin{align}
m(x,t)&=\lim_{\lambda\rightarrow\infty }(\lambda M )_{12}=\lim_{z\rightarrow\infty }(z N )_{12} =(N_1)_{12}\nonumber\\
& = \frac{\beta_{12} }{i\sqrt{8t}}+O(t^{-3/4})= \frac{ 1 }{ \sqrt{2t}}  \alpha(z_0)  e^{i(4tz_{0}^{2}-\nu(z_0)\log8t)}+O(t^{-3/4}),\label{9.10}
\end{align}
  which leads to
$$|m(x,t)|^2\sim \frac{\nu}{8t}, \ \ t\rightarrow\infty.$$
Further calculating gives
\begin{equation}\label{9.11}
\begin{aligned}
\int_{-\infty}^x|m(x',t)|^2\mathrm{d}x'&\sim \int_{-\infty}^x\frac{\nu}{8t}\mathrm{d}x'
=-\frac{1}{2\pi}\int_{\lambda_{0}}^{+\infty}\log(1-|r(\lambda)|^{2})(a\lambda-b)\mathrm{d}\lambda.
\end{aligned}
\end{equation}
with $z_0=\lambda_0(a\lambda_0-2b)$. Finally, substituting (\ref{9.10}) and (\ref{9.11}) into ({\ref{419}) yields
\begin{equation}
q(x,t)=\frac{1}{\sqrt{t}}\alpha(z_0)e^{i(4tz_{0}^{2}-\nu(z_0)\log8t)}e^{-\frac{2ia}{\pi}
\int^{+\infty}_{\lambda_0}\log(1-|r(\lambda)|^2)(a\lambda-b)\mathrm{d}\lambda}+O(t^{-\frac{3}{4}}),
\end{equation}
which leads to  our main results.

\begin{thm} \label{thm1.1}
For initial data $q_0(x)\in H^{2,2}(\mathbb{R}) $ in the Sobolev space
 with reflection coefficient $\rho_1(z), \rho_2(z)\in H_1^{2,2}(\mathbb{R})$,
 then  long-time  asymptotic of  the solution for  the mixed NLS  equation  (\ref{MNLS})  is given by
\begin{equation}
q(x,t)=\frac{1}{\sqrt{t}}\alpha(z_0)e^{i(4tz_{0}^{2}-\nu(z_0)\log8t)}e^{-\frac{2ia}{\pi}\int^{+\infty}_{\lambda_0}\log(1-|r(\lambda)|^2)(a\lambda-b)\mathrm{d}\lambda}+O(t^{-3/4}),\nonumber
\end{equation}
where
\begin{equation}
\alpha(z_{0})=\frac{\pi^{\frac{1}{2}}e^{\frac{\pi\nu}{2}}e^{\frac{i\pi}{4}}e^{2\chi(z_0)}}{(a\lambda_0-2b) r(\lambda_0)\Gamma(-a)},
\label{alpha}
\end{equation}
whose modulus is
\begin{equation}
\vert \alpha(z_0)\vert^2=-\frac{1}{4\pi}\ln(1-|r(\lambda_0)|^2),
\end{equation}
and angle  is
\begin{equation}
\begin{aligned}
\arg \alpha(z_0)&=\frac{1}{\pi}\int_{-\infty}^{z_0}\log|z_0-\lambda(a\lambda-2b)|\mathrm{d}\log(1-|r(\lambda)|^{2})\\
&+\frac{\pi}{4}-\arg [(a\lambda_0-2b)r(\lambda_{0})]+\arg \Gamma(i\nu).
\end{aligned}
\end{equation}
\end{thm}

\noindent\textbf{Acknowledgements}

This work is supported by  the National Natural Science
Foundation of China (Grant No. 11671095,  51879045).


\begin{thebibliography}{99}
\bibitem{ago1979} M. Wadati, K. Konno and Y. H. Ichikawa,
\newblock {A Generalization of Inverse Scattering Method},
\newblock {\em J. Phys. A}, 36(1979), 1965-1966.

\bibitem{imf1980} T. Kawata, J. Sakai  and N. Kobayashi,
\newblock {Inverse method for the mixed nonlinear schr\"{o}dinger-equation and solition-solutions},
\newblock {\em J. Phys. Soc. Jpn.}, 48(1980), 1371-1379.

\bibitem{Stiassnie1} M. Stiassnie,
	\newblock {Note on the modified nonlinear schr\"{o}dinger equation for deep water waves},
	\newblock {\em Wave Motion}, 6(1984), 431-433.

\bibitem{Mio1} K. Mio, T. Ogino, K. Minami and  S. Takeda,
	\newblock {Modified nonlinear Schr\"{o}dinger equation for Alfv\'{e}n waves propagating along the magnetic field in cold plasmas},
	\newblock {\em J. Phys. Soc. Jpn.}, 41(1976), 265-271.

\bibitem{Al} G. P. Agrawal,
	\newblock {Nonlinear Fiber Optics, 4th ed.},
	 Academic Press, Boston, 2007.

\bibitem{Yang} J. K. Yang,
	\newblock {Nonlinear Waves in Integrable and Nonintegrable Systems},
	SIAM, Philadelphia, 2010.

\bibitem{Nakatsuka1} H. Nakatsuka, D. Grischkowsky and  A. C. Balant,
	\newblock {Nonlinear picosecond-Pulse propagation through optical fibers with positive group velocity dispersion},
	\newblock {\em Phys. Rev.  Lett.}, 47(1981), 910-913.

\bibitem{Tzoar1} N. Tzoar and  M. Jain,
	\newblock {Self-phase modulation in long-geometry optical waveguides},
	\newblock {\em Phys. Rev. A}, 23(1981), 1266-1270.

\bibitem{RN13} L. Brizhik, A. Eremko and  B. Piette,
	\newblock {Solutions of a D-dimensional modified nonlinear Schr\"{o}dinger equation},
	\newblock {\em Nonlinearity}, 13(2003),  1481-1497.
	
\bibitem{RN5}G. Biondini,  G. Kovacic,
\newblock{Inverse scattering transform for the focusing nonlinear Schrodinger equation with nonzero boundary conditions},
\newblock {\em Journal of Mathematical Physics}, 55(2014), 339-351.

\bibitem{RN6}F. Demontis, B. Prinari,   C. Van Der Mee,
\newblock{The inverse scattering transform for the defocusing nonlinear Schrodinger equations with nonzero boundary conditions},
\newblock{\em Studies in Applied Mathematics},  131(2013),  1-40.


\bibitem{RN4} D. J. Kaup, A. C. Newell,
 \newblock{An exact solution for a derivative nonlinear Schrodinger equation},
 \newblock{\em Journal of Mathematical Physics}, 19(1978), 798-801.

\bibitem{Maimistov1}   A. I. Maimistov,
	\newblock {Evolution of solitary waves which are approximately solitons of a nonlinear Schr\"{o}dinger equation},
	\newblock {\em J. Exp. Theor. Phys.}, 77(1993), 727-731.

\bibitem{pso1985} A. Roy Chowdhury, S. Paul, and S. Sen,
    \newblock {Periodic solutions of the mixed nonlinear Schr\"{o}dinger equation},
    \newblock {\em Phys. Rev.  D}, 32(1985), 3233-3237.

\bibitem{oss1991} B. L. Guo, S. B. Tan,
    \newblock {On smooth solutions to the initial value problem for the mixed
nonlinear Schr\"{o}dinger equations},
    \newblock {\em Proc.   Royal Soc.}, 119A(1991), 31-45.

\bibitem{ows1994} S. B. Tian, L.H. Zhang,
    \newblock {On a weak solution of the Mixed Nonlinear Schr\"{o}dinger Equations},
    \newblock {\em J. Math. Anal.   Appl.}, 182(1994), 409-421.

\bibitem{bsf2004} S. B. Tian,
    \newblock {Blow-up Solutions for Mixed Nonlinear Schr\"{o}dinger Equations},
    \newblock {\em Acta Mathematica Sinica}, 20(2004), 115-124.

\bibitem{mfd2015} X.  L\"{u},
    \newblock {Madelung fluid description on a generalized mixed nonlinear Schr\"{o}dinger equation},
    \newblock {\em Nonlinear Dyn}, 81(2015), 239-247.

\bibitem{sbf2013} X.  L\"{u},
    \newblock {Soliton behavior for a generalized mixed nonlinear Schr\"{o}dinger model with N-fold
Darboux transformation},
    \newblock {\em Chaos}, 23(2013), 033137.

\bibitem{hs2015} J. S. He, S. W.  Xu, and Y. Cheng,
    \newblock {The rational solutions of the mixed nonlinear Schr\"{o}dinger equation},
    \newblock {\em AIP Advances}, 5(2015), 017105.

  

\bibitem{osf}  M. Tsutsumi, I. Fukuda,
	\newblock {On solutions of the derivative nonlinear Schr\"{o}dinger equation. Existence and uniqueness theorem},
	\newblock {\em Funkcial. Ekvac.}, 23 (1980), 259-277.

\bibitem{ht1999}  H. Takaoka,
	\newblock{Well-posedness for the one-dimensional nonlinear Schr\"{o}dinger equation with the derivative nonlinearity},
	\newblock {\em  Adv. Differential Equations}, 4 (1999), 561-580.

\bibitem{nt1992}  N. Hayashi, T. Ozawa,
	\newblock {On the derivative nonlinear Schr\"{o}dinger equation},
	\newblock {\em Phys. D}, 55 (1992), 14-36.

\bibitem{jm2002}  J. Colliander, M. Keel, G. Staffilani, H. Takaoka, and T. Tao,
	\newblock { A refined global well-posedness result
for Schr\"{o}dinger equations with derivative},
	\newblock {\em SIAM J. Math. Anal.}, 34 (2002), 64-86.

\bibitem{zy2017}  Z. Guo, Y. Wu,
	\newblock {Global well-posedness for the derivative nonlinear Schr\"{o}dinger equation
in $H^{1/2}(\mathbb{R}$)},
	\newblock {\em  Discrete Contin. Dyn. Syst.}, 37 (2017), 257-264.

\bibitem{yw2015}  Y. Wu,
	\newblock {Global well-posedness on the derivative nonlinear Schr\"{o}dinger equation},
	\newblock {\em  Anal. PDE}, 8 (2015), 1101-1112.

\bibitem{tdn} R. Jenkins, J. Liu, P. Perry, C. Sulem,
	\newblock {The derivative nonlinear schr\"{o}dinger euqation:
global well-posedness and soliton resolution},
	\newblock {\em Quarterly of applied mathematicas},  78(2020), 33-73.

\bibitem{gwp33} R. Jenkins, J. Liu, P. Perry, C. Sulem,
	\newblock {Global well-posedness for the derivative nonlinear
Schr\"{o}dinger equation},
	\newblock {\em Math. Phys.}, 363(2018), 1003-1049.

\bibitem{IMRN2006}   K. D. T-R McLaughlin, P. D. Miller,
\newblock{$\bar{\partial}$-steepest descent method and the asymptotic behavior of polynomials orthogonal and exponentially varying nonanalytic weights},
	\newblock { \em Int. Math. Res. Not.}, IMRN (ISSN1687-3017) (2006) 48673.

\bibitem{lta29}  M. Dieng, K.D.T-R McLaughlin,
	\newblock {Long-time asymptotics for the NLS equation via $\bar{\partial}$ methods},
	\newblock {arXiv:0805.2807v1.}

\bibitem{lta31}  M. Borghese, R. Jenkins, K. K.D.T-R McLaughlin,
	\newblock {Long time asymptotics behavior of the focusing nonlinear Schr\"{o}dinger equation},
	\newblock {\em Ann. I. H. Poincar\'{e}}, AN 35(2018), 887-920.


\bibitem{ANN2018}   J. Liu, P.A. Perry, C. Sulem,
	\newblock { Long-time behavior of solutions to the derivative nonlinear Schr\"{o}inger equation for soliton-free initial data},
	\newblock {\em Ann. I. H. Poincar\'{e}}, AN 35(2018), 217-265.

 \bibitem{Ma2019}R. H. Ma, E.G. Fan,	
\newblock { Long time asymptotics behavior of the focusing nonlinear Kundu-Eckhaus equation}, arXiv:1912.01425v1, 2019

 \bibitem{yf2019} 	Y. L. Yang, E.G. Fan,
\newblock {Long-time asymptotic behavior of the modified Schrodinger equation via Dbar-steepest descent method}, arXiv:2980918, 2019


\bibitem{CF}  Q. Y. Cheng, E. G. Fan,
\newblock {Long-time asymptotics for a mixed nonlinear Schr\"{o}dinger equation with the Schwartz initial data},
\newblock {\em J. Math. Anal. Appl}, 489(2020), 124188.1-24.

\bibitem{lts22}  P. Deift, X. Zhou,
	\newblock {Long-time asymptotics forsolutions ofthe NLS equation with initial data in a weighted
Sobolev space},
	\newblock {\em Comm. Pure and Appl. Math.}, LVI(2003), 1029-1077.

\end{thebibliography}
\end{document}